\documentclass[]{article}
\usepackage{fullpage} 
\usepackage{amsmath, amsthm, breqn, graphicx,mathtools}
\usepackage{scalerel}
\theoremstyle{definition}

\newtheorem*{lemma*}{lemma}

\newcommand\scalemath[2]{\scalebox{#1}{\mbox{\ensuremath{\displaystyle #2}}}}
\renewcommand{\vec}[1]{\mathbf{#1}}

\usepackage{titlesec}
\titleformat*{\section}{\filcenter}

\title{Determinant of a Sum of Certain Kronecker Products}
\author{Dwight Nwaigwe}

\begin{document}

\maketitle

\begin{abstract}
We compute the determinant of $\sum_{n=1}^{N} \vec{A}^{(n)} \otimes 
\vec{B}^{(n)}$, 
where $\vec{A}^{(n)}$ is square and 
${\vec{B}^{(n)}=\vec{x}^{(n)}{\vec{y}^{(n)}}^T}$ where 
$\vec{x}^{(n)}$ and $\vec{y}^{(n)}$ have length $N$. 
\end{abstract}

\section*{Introduction}
Kronecker products appear in a variety of applications 
\cite{benner}, \cite{benzing}. For square matrices $\vec{A}, \vec{B}$ 
the 
determinant 
of the Kronecker product $\vec{A} \otimes \vec{B}$ is given by $\det(\vec{A})^N 
\det(\vec{B})^F$, where $F$ is the size of $\vec{A}$ and $N$ the size of 
$\vec{B}$ 
\cite{horn}. 
Typically there are no general formulas for more 
complicated expressions involving Kronecker products, such as sums. In this 
article we compute the determinant of a sum of certain Kronecker products as 
described in the next section.

\section*{Result}

\begin{lemma*}
Let $\{\vec{A}^{(n)}\}$ be $F$ by $F$ matrices and 
$\vec{B}^{(n)}=\vec{x}^{(n)} {\vec{y}^{(n)}}^T$ where 
$\vec{x}^{(n)}$ is the $n^{th}$ column of some $N$ by $N$ matrix $\vec{X}$, and 
$\vec{y}^{(n)}$ 
is the $n^{th}$ column of some $N$ by $N$ matrix $\vec{Y}$. Also, let

\begin{equation}
\vec{G}= \sum_{n=1}^{N} \vec{A}^{(n)} \otimes \vec{B}^{(n)}.
\end{equation}
Then, the following equation holds:

\begin{equation}
\label{eq:result}
\det(\vec{G})=\left(\prod_{n=1}^{N} \det(\vec{A}^{(n)}) \right) 
{\det(\vec{X})}^{F}{\det(\vec{Y})}^{F}.
\end{equation}
\end{lemma*}

\begin{proof}

$\det(\vec{G})$ is a multi-linear function of its columns, therefore it can be 
expressed a sum of determinants where each determinant corresponds to a 
particular arrangement of columns belonging to the various $\vec{A}^{(n)} 
\otimes \vec{B}^{(n)}$.  Denote the set of such corresponding matrices as 
$\mathcal{O}$. Then $\det(\vec{G})= \sum_{\vec{C} \in \mathcal{O}} 
\det(\vec{C})$.  Because 
$\vec{B}^{(n)}$ is rank 1, only certain matrices contribute to the sum and as a 
result we can redefine $\mathcal{O}$ to consist of these matrices. Let $\vec{C} 
\in \mathcal{O}$.  The column entries of $\vec{C}$ are such that: each of the 
first 
$N$ columns come from a different $\vec{A}^{(n)} \otimes \vec{B}^{(n)}$;  each 
of the columns $N+1$ to $2N$ come from a different 
$\vec{A}^{(n)} \otimes \vec{B}^{(n)}$; and so forth for the 
$(N-1)F+1$ to $NF$ columns. For this reason, let $S_N$ be the permutation group 
on $N$ elements, and let $(\gamma^1,\gamma^2,\dots \gamma^F) \in (S_N)^F$. 
Every $\vec{C} \in \mathcal{O}$ 
can be identified by some $(\gamma^1,\gamma^2,\dots \gamma^F)$. By definition 
of $\gamma^k$, we have that $\vec{A}^{\gamma^k(i)}=\vec{A}^{(n)}$ for some $i$, 
and $n$, and 
$\vec{B}^{\gamma^k(i)}=\vec{B}^{(n)}$ for some $i$ and $n$. It is possible 
that $(\gamma^1,\gamma^2,\dots 
\gamma^F)$ may not have distinct elements. Also, we now write 
$\vec{C}^{ \scaleto{\left(\gamma^1,\gamma^2,\dots 
\gamma^F\right)}{8pt} }$ instead of $\vec{C}$ to emphasize the identification 
of $\vec{C} \in \mathcal{O}$ with $(S_N)^F$.

For a given  $\vec{C}^{ \scaleto{\left(\gamma^1,\gamma^2,\dots 
\gamma^F\right)}{8pt} } \in \mathcal{O}$, it is sufficient to only consider 
the case of a block diagonal structure since contributions from the non-block 
diagonal terms in $\vec{C}^{ \scaleto{\left(\gamma^1,\gamma^2,\dots 
		\gamma^F\right)}{8pt} }$ can be expressed as contributions from a block 
diagonal 
matrix up to sign. For the purpose of the proof, we thus only calculate 
contributions from 
the block diagonal terms. Let $\vec{C}_{diag}^{ 
\scaleto{\left(\gamma^1,\gamma^2,\dots \gamma^F\right)}{8pt} }$ be the 
submatrix of  $\vec{C}^{ 
	\scaleto{\left(\gamma^1,\gamma^2,\dots \gamma^F\right)}{8pt} }$ determined 
	by replacing the asterisks with 0's in the 
below 
matrix: We now calculate 
$\det(\vec{C}_{diag}^{\scaleto{\left(\gamma^1,\gamma^2,\dots 
		\gamma^F\right)}{8pt} })$ using the 
Leibniz rule.  

\begin{figure}
\[
\hspace*{-2cm} 
\scalemath{0.8}{
\left[ 
\begin{array}{c@{}c@{}c}
\left[\begin{array}{cccc}  
\vec{A}^{\gamma^1(1)}_{11} \vec{B}^{\gamma^1(1)}_{11} & 
\vec{A}^{\gamma^1(2)}_{11} \vec{B}^{\gamma^1(2)}_{12} &\ldots & 
\vec{A}^{\gamma^1(n)}_{11}  
\vec{B}^{\gamma^1(n)}_{1N}  \\
\vec{A}^{\gamma^1(1)}_{11} \vec{B}^{\gamma^1(1)}_{21} & 
\vec{A}^{\gamma^1(2)}_{11} \vec{B}^{\gamma^1(2)}_{22} &\ldots & 
\vec{A}^{\gamma^1(n)}_{11}  
\vec{B}^{\gamma^1(n)}_{2N}  \\
\vdots & \vdots &\ddots &  \\
\vec{A}^{\gamma^1(1)}_{11} \vec{B}^{\gamma^1(1)}_{N1} & 
\vec{A}^{\gamma^1(2)}_{11} \vec{B}^{\gamma^1(2)}_{N2} &\ldots & 
\vec{A}^{\gamma^1(n)}_{11}  
\vec{B}^{\gamma^1(n)}_{NN}  \\
\end{array}\right]    & \mathbf{*} & \mathbf{*} \\
\mathbf{*} & \left[\begin{array}{cccc}
\vec{A}^{\gamma^2(1)}_{22} \vec{B}^{\gamma^2(1)}_{11} & 
\vec{A}^{\gamma^2(2)}_{22} \vec{B}^{\gamma^2(2)}_{12} &\ldots & 
\vec{A}^{\gamma^2(n)}_{22}  
\vec{B^}{\gamma^2(n)}_{1N}  \\
\vec{A}^{\gamma^2(1)}_{22} \vec{B}^{\gamma^2(1)}_{21} & 
\vec{A}^{\gamma^2(2)}_{22} \vec{B}^{\gamma^2(2)}_{22} &\ldots & 
\vec{A}^{\gamma^2(n)}_{22}  
\vec{B}^{\gamma^2(n)}_{2N}  \\
\vdots & \vdots &\ddots &  \\
\vec{A}^{\gamma^2(1)}_{22} \vec{B}^{\gamma^2(1)}_{N1} & 
\vec{A}^{\gamma^2(2)}_{22} \vec{B}^{\gamma^2(2)}_{N2} &\ldots & 
\vec{A}^{\gamma^2(n)}_{22}  
\vec{B}^{\gamma^2(n)}_{NN}  \\
\end{array}\right]  & \mathbf{*}\\
\mathbf{*} & \mathbf{*} & \left[\begin{array}{cccc}
\vec{A}^{\gamma^2(1)}_{NN} \vec{B}^{\gamma^2(1)}_{11} & 
\vec{A}^{\gamma^2(2)}_{NN} \vec{B}^{\gamma^2(2)}_{12} &\ldots & 
\vec{A}^{\gamma^2(n)}_{NN}  
\vec{B}^{\gamma^2(n)}_{1N}  \\
\vec{A}^{\gamma^2(1)}_{NN} \vec{B}^{\gamma^2(1)}_{21} & 
\vec{A}^{\gamma^2(2)}_{NN} \vec{B}^{\gamma^2(2)}_{22} &\ldots & 
\vec{A}^{\gamma^2(n)}_{NN}  
\vec{B}^{\gamma^2(n)}_{2N}  \\
\vdots & \vdots &\ddots &  \\
\vec{A}^{\gamma^2(1)}_{NN} \vec{B}^{\gamma^2(1)}_{N1} & 
\vec{A}^{\gamma^2(2)}_{NN}\vec{B}^{\gamma^2(2)}_{N2} &\ldots & 
\vec{A}^{\gamma^2(n)}_{NN}  
\vec{B}^{\gamma^2(n)}_{NN}  \\
\end{array}\right] 
\end{array}\right] } \\
\]
\caption{In computing the determinant in our proof, we only do so for the 
blocks on diagonals. The 
off-diagonal terms can be put into a block diagonal form up to sign and 
accounted for later.}
\end{figure}

If $\{\vec{G_i}\}$ are the diagonals of a block diagonal matrix, then $ \det 
\left( \vec{G_ 1}\bigoplus 
\vec{G_2} \bigoplus \ldots \vec{G_M} \right) = \prod_{i=1}^{M} \det(\vec{G_i})  
$. Applying this 
to $\det( \vec{C}_{diag}^{ \scaleto{\left(\gamma^1,\gamma^2,\dots 
\gamma^F\right)}{8pt} })$ gives


\begin{align}
 \label{eq:C_diag}
\hspace*{-3.5cm} 
& \det(\vec{C}_{diag}^{ \scaleto{\left(\gamma^1,\gamma^2,\dots 
\gamma^F\right)}{8pt} 
}) \\
& = \prod_{i=1}^{F} det(\bf{block}_i) \\
& = det\left(  
(\vec{A}^{\gamma^1(1)}_{11} \vec{A}^{\gamma^1(2)}_{11} \dots 
\vec{A}^{\gamma^1(N)}_{11}) 
[\vec{B}^{\gamma^1(1)}_{1} \vec{B}^{\gamma^1(2)}_{2}\dots 
\vec{B}^{\gamma^1(N)}_{N} ]         
\right) \boldsymbol{\cdot} \nonumber \\  
& \quad det\left(  (\vec{A}^{\gamma^2(1)}_{11} \vec{A}^{\gamma^2(2)}_{11} \dots 
\vec{A}^{\gamma^2(N)}_{11}) [\vec{B}^{\gamma^2(1)}_{1} 
\vec{B}^{\gamma^2(2)}_{2} \dots 
\vec{B}^{\gamma^2(N)}_{N} ]  \right) \boldsymbol{\cdot}  \\
& \vdots \\
& \quad det\left(  
(\vec{A}^{\gamma^F(1)}_{11} \vec{A}^{\gamma^F(2)}_{11} \dots  
\vec{A}^{\gamma^F(N)}_{11}) 
[\vec{B}^{\gamma^F(1)}_{1} \vec{B}^{\gamma^F(2)}_{2} \dots 
\vec{B}^{\gamma^F(N)}_{N} ]         
\right) \nonumber \\
& =(\vec{A}_{11}^1 \vec{A}_{11}^2 \dots \vec{A}_{11}^F) (\vec{A}_{22}^1 
\vec{A}_{22}^2 \dots \vec{A}_{22}^F) \dots 
(\vec{A}_{FF}^1 \vec{A}_{FF}^2  \dots \vec{A}_{FF}^F)\prod_{i=1}^{F} 
\det\left([\vec{B}^{\gamma^i(1)}_{1}  \vec{B}^{\gamma^i(2)}_{2} \dots 
\vec{B}^{\gamma^i(N)}_{N}] 
\right)
\end{align}
where $[\vec{B}^{\gamma^i(1)}_{1} 
\vec{B}^{\gamma^i(2)}_{2}\dots \vec{B}^{\gamma^i(N)}_{N} ]$ 
denotes the matrix composed of the column vectors $\vec{B}^{\gamma^i(k)}_{j}$, 
$1 
\leq j \leq N$, with $j$ referring to column index. We have that

\begin{eqnarray}
 \label{eq:detb}
det\left( [ \vec{B}^{\gamma^i(1)}_{1} 
\vec{B}^{\gamma^i(2)}_{2}\dots \vec{B}^{\gamma^i(N)}_{N} ] 
\right) = \left(\vec{Y}_{1 \gamma^i(1)} \vec{Y}_{2 \gamma^i(2)}\dots \vec{Y}_{N 
\gamma^i(N)} 
\right)  
\det(\vec{Z}), 
\end{eqnarray}
where 
\[
\det(\vec{Z})=
\begin{vmatrix} 
\vec{X}_{1\gamma^i(1)} & \vec{X}_{1\gamma^i(2)} & \dots \vec{X}_{1\gamma^i(N)} 
\\
\vec{X}_{2\gamma^i(1)} & \vec{X}_{2\gamma^i(2)} & \dots \vec{X}_{2\gamma^i(N)} 
\\
\vdots & \vdots & \ddots & \\
\vec{X}_{N\gamma^i(1)} & \vec{X}_{N\gamma^i(2)} & \dots \vec{X}_{N\gamma^i(N)} 
\\
\end{vmatrix}
.\]
Clearly,

\begin{equation}
\det(\vec{Z})=\det(\vec{X}) sgn(\gamma^i) 
\end{equation}
which combined with \eqref{eq:detb} gives 

\begin{align}
\label{eq:proddet}
 \prod_{i=1}^{F} \det[\vec{B}^{\gamma^i(1)}_{1}  \vec{B}^{\gamma^i(2)}_{2} 
 \dots 
 \vec{B}^{\gamma^i(N)}_{N}] 
 = (\det(\vec{X}))^F  \prod_{i=1}^{F} (\vec{Y}_{1 \gamma^i(1)} \vec{Y}_{2 
 \gamma^i(2)} \dots  
 \vec{Y}_{N \gamma^i(N)} ) sgn(\gamma^i). 
\end{align} 
Plugging in \eqref{eq:proddet} into \eqref{eq:C_diag} gives

\begin{align}
\det(\vec{C}_{diag}^{ \scaleto{\left(\gamma^1,\gamma^2,\dots 
\gamma^F\right)}{8pt} }) 
\\
&= (\vec{A}_{11}^1 A_{11}^2 \dots \vec{A}_{11}^F) (\vec{A}_{22}^1 
\vec{A}_{22}^2 \dots \vec{A}_{22}^F) \dots 
(\vec{A}_{FF}^1 \vec{A}_{FF}^2 \dots \vec{A}_{FF}^F) (\det(\vec{X}))^F 
\boldsymbol{\cdot} \nonumber \\
& \prod_{i=1}^{F} (\vec{Y}_{1 \gamma^i(1)} \vec{Y}_{2 \gamma^i(2)} \dots 
\vec{Y}_{N \gamma^i(N)} ) sgn(\gamma^i).
\end{align}
To get the contributions from all block-diagonal matrices, we must sum over all 
possible sequences of $(\gamma^1,\gamma^2,\dots \gamma^F) \in (S_N)^F$. Doing 
this we have 

\begin{align}
\label{eq:sumdiag}
 \sum_{ (\gamma^1,\gamma^2,\dots \gamma^F) \in (S_N)^F  } \det(\vec{C}_{diag}^{ 
\scaleto{\left(\gamma^1,\gamma^2,\dots \gamma^F\right)}{8pt} }) \\ \nonumber  
&= (\vec{A}_{11}^1 
\vec{A}_{11}^2 \dots \vec{A}_{11}^F) (\vec{A}_{22}^1 \vec{A}_{22}^2 \dots 
\vec{A}_{22}^F) \dots (\vec{A}_{FF}^1 
\vec{A}_{FF}^2 \dots \vec{A}_{FF}^F) (\det(\vec{X}))^F \boldsymbol{\cdot} \\ 
\nonumber  & 
\nonumber 
\left( \sum_{ (\gamma^1,\gamma^2,\dots \gamma^F) \in (S_N)^F  } 
\prod_{i=1}^{F} (\vec{Y}_{1 \gamma^i(1)} \vec{Y}_{2 \gamma^i(2)} \dots 
\vec{Y}_{N \gamma^i(N)} ) 
sgn(\gamma^i)\right). 
\end{align}
However, the term $\left( \sum_{ (\gamma^1,\gamma^2,\dots\gamma^F) \in (S_N)^F  
} \prod_{i=1}^{F} (\vec{Y}_{1 \gamma^i(1)} \vec{Y}_{2 \gamma^i(2)} \dots 
\vec{Y}_{N \gamma^i(N)} 
) sgn(\gamma^i) \right)$ is equal to $(\det(\vec{Y}))^F  $, so that 
\eqref{eq:sumdiag} 
becomes

\begin{align}
\label{total-off}
 \sum_{ (\gamma^1,\gamma^2,\dots \gamma^F) \in (S_N)^F  } \det(\vec{C}_{diag}^{ 
 \scaleto{\left(\gamma^1,\gamma^2,\dots \gamma^F\right)}{8pt} })\\
 & = (\vec{A}_{11}^1 \vec{A}_{11}^2 \dots \vec{A}_{11}^F) (\vec{A}_{22}^1 
 \vec{A}_{22}^2 \dots \vec{A}_{22}^F) 
 \dots  (\vec{A}_{FF}^1 \vec{A}_{FF}^2 \dots \vec{A}_{FF}^F) 
 (\det(\vec{X}))^{F} (\det(\vec{Y}))^{F} \nonumber
\end{align}
Adding the contributions from the non-block diagonal terms to 
\eqref{total-off}, \eqref{eq:result} is proved.
\end{proof}

\bibliographystyle{plain}
\bibliography{proof_of_first_half_of_determinant_theorem}

\end{document}